\documentclass[twoside,a4paper,11pt]{article}
\usepackage[top=4cm, bottom=4cm, left=3cm, right=3cm]{geometry}

\usepackage{appendix}
\usepackage[sort,square,numbers]{natbib}

\renewcommand{\phi}{\varphi}
\usepackage{macros_sink}

\usepackage{fancyhdr}
\title{Asymptotics for semi-discrete entropic optimal transport}

 \author{
         Jason M. Altschuler\footnote{Supported in part by NSF Graduate Research Fellowship 1122374 and a TwoSigma PhD fellowship.}\\
         MIT\\
         \texttt{jasonalt@mit.edu}
         \and
         Jonathan Niles-Weed%
         \footnote{Supported in part by NSF grant DMS-2015291.} \\
         NYU \\
         \texttt{jnw@cims.nyu.edu}
         \and
         Austin J. Stromme\footnote{Supported in part by NDSEG Fellowship F-6749924378.} \\
         MIT \\
         \texttt{astromme@mit.edu}
}

\begin{document}
\maketitle

\begin{abstract}
We compute exact second-order asymptotics for the cost of an optimal solution to the entropic optimal transport problem in the continuous-to-discrete, or \emph{semi-discrete}, setting.
In contrast to the discrete-discrete or continuous-continuous case, we show that the first-order term in this expansion vanishes but the second-order term does not, so that in the semi-discrete setting the difference in cost between the unregularized and regularized solution is \emph{quadratic} in the inverse regularization parameter, with a leading constant that depends explicitly on the value of the density at the points of discontinuity of the optimal unregularized map between the measures.
We develop these results by proving new pointwise convergence rates of the solutions to the dual problem, which may be of independent interest.
\end{abstract}

\section{Introduction}
The entropically regularized optimal transportation problem, originally inspired by a thought experiment of Schr\"odinger~\cite{Sch31} and the subject of a great deal of recent interest in probability~\cite{Fol88,Mik04}, statistics~\cite{RigWee18,GenChiBac19,MenNil19,ChiRouLeg20} and machine learning~\cite{GenPeyCut18,cuturi2013sinkhorn}, is an optimization problem which seeks a coupling between two probability measures that minimizes the transport cost between them, subject to an additional entropic penalty.
Specifically, given Borel probability measures $\mu$ and $\nu$ on $\R^d$ with finite second moment and $\eta > 0$, the problem reads
\begin{equation}\label{eqn:eot}
	\inf_{\pi \in \Pi(\mu, \nu)}  \E_{\pi}[\| x - y \|^2 ] + \frac{1}{\eta} \KL{\pi}{\mu \otimes \nu}\,,
\end{equation}
where $\Pi(\mu, \nu)$ denotes the set of couplings of $\mu$ and $\nu$ and $\KL{\cdot}{\cdot}$ denotes the Kullback--Leibler divergence or relative entropy, defined by
\begin{equation*}
	\KL{\pi}{\rho} := \begin{cases}
		\int \log \frac{d \pi}{d \rho}(x) d \pi(x) & \pi \ll \rho \\
		+ \infty & \text{otherwise.}
	\end{cases}
\end{equation*}

Recent interest in~\eqref{eqn:eot} has been driven by the fact that, as $\eta \to \infty$, the solution $\pi_\eta$ to~\eqref{eqn:eot} approaches the solution $\pi^*$ to the unregularized optimal transport problem~\cite{Leo12,carlier2017convergence},
\begin{equation}\label{eqn:wass}
	\inf_{\pi \in \Pi(\mu, \nu)} \E_\pi [\|x - y\|^2]\,,
\end{equation}
which defines the squared Wasserstein distance $W_2^2(\mu, \nu)$~\cite{villani2008optimal}.
In statistics and machine learning applications, it has been recognized that~\eqref{eqn:eot} represents a computationally and statistically attractive proxy for~\eqref{eqn:wass}.
Statistically, the entropically regularized problem offers improved sample complexity~\cite{GenChiBac19} and cleaner limit laws~\cite{MenNil19} than its unregularized counterpart; computationally, the strict convexity of~\eqref{eqn:eot} opens the door to much faster algorithms~\cite{cuturi2013sinkhorn,altschuler2017near}.

The importance of the $\eta \to \infty$ limit has spurred a line of work which seeks to quantify the speed of convergence of $\pi_\eta \to \pi^*$ and to develop higher-order asymptotics in the $\eta \to \infty$ regime.
Of particular interest is the \emph{suboptimality} of the entropically regularized solution:
$$
\E_{\pi_{\eta}}[\|x - y\|^2] - \E_{\pi^*}[\|x - y\|^2]\,.
$$
This quantity measures the suitability of $\pi_\eta$ as an approximation for $\pi^*$, and giving precise bounds is essential for statistical and computational applications.

Two cases are well understood, with vastly different rates: when $\mu$ and $\nu$ are both finitely supported, then it is known that the difference in cost approaches zero exponentially fast as $\eta \to \infty$~\cite{cominetti1994asymptotic,weed2018explicit}.
On the other hand, when $\mu$ and $\nu$ are absolutely continuous measures with bounded, compactly supported densities, then precise asymptotics to second order are known for the cost \emph{including the entropic term}~\cite{ConTam21,Pal19,ErbMaaRen15,ChiRouLeg20}: as $\eta \to \infty$,
\begin{multline}\label{eqn:second-order}
 \E_{\pi_\eta}[\| x - y \|^2 ] + \frac{1}{\eta} \KL{\pi_\eta}{\mu \otimes \nu} = W_2^2(\mu, \nu) - \frac{d}{2 \eta} \log \left(\frac \pi \eta\right) + \frac{1}{2\eta} (h(\mu) + h(\nu)) \\ + \frac{1}{16 \eta^2} I(\mu, \nu) + o(\eta^{-2})\,, 
\end{multline}
where for a probability measure $\mu$ with density $\mu(\cdot)$ with respect to the Lebesgue measure we write
\begin{equation*}
h(\mu) := - \int \log(\mu(x)) \mu(x) d x
\end{equation*}
for the entropy relative to the Lebesgue measure, and where $I$ is the integrated Fisher information along the Wasserstein geodesic connecting $\mu$ to $\nu$.
It does not seem possible to extract asymptotics for the cost $\E_{\pi_\eta}[\| x - y \|^2 ]$  directly from~\eqref{eqn:second-order}; however, it is easy to show that in general for absolutely continuous $\mu$ and $\nu$, the suboptimality is linear in $\eta^{-1}$.
For example, when $\mu$ and $\nu$ are Gaussian measures on $\R$, it can be checked directly that
\begin{equation*}
\E_{\pi_{\eta}}[\|x - y\|^2] - \E_{\pi^*}[\|x - y\|^2] = \frac{1}{2\eta} + o(\eta^{-1})\,.
\end{equation*}

The large gulf between these convergence rates---exponential for finitely supported measures, linear in $\eta^{-1}$ for absolutely continuous measures---raises the question of which of the two behaviors should be expected in general.
As a first step towards understanding this question, we study a situation between these two extremes: the \emph{semi-discrete} case, in which one measure is absolutely continuous and the other is finitely supported.
This setting is important for both theoretical and practical reasons, but prior work gives no hint of how the suboptimality in the semi-discrete case should behave.
Should one expect to recover the exponential rate or the linear rate?

A simulation in one dimension, where all the quantities are explicit, shows, perhaps surprisingly, that the rate in the semi-discrete case is something else entirely.
Figure~\ref{fig:sttb} plots the suboptimality for two different one-dimensional examples as $\eta$ varies, one where $\mu$ is the Gaussian density, and the other where $\mu$ is the Laplacian density.
For both experiments, we take $\nu$ to be a discrete measure, uniform on $\{-1, +1\}$.
The apparent result is that in both cases, the suboptimality is neither linear nor exponential but \emph{quadratic} in $\eta^{-1}$.
Moreover, the very careful reader will note that the asymptotic suboptimality appears to agree with $\frac{\pi^2\mu(0)}{24}  \eta^{-2}$, where $\mu(0)$ is the value of the density $\mu$ at the origin, which is also the point at which the optimal \emph{unregularized} map from $\mu$ to $\nu$ changes value from $-1$ to $+1$.
We give a full exposition of this example in Section~\ref{sec:sttb}.
\begin{figure}
	\begin{subfigure}{0.45\linewidth} \centering
		\includegraphics[scale=0.38]{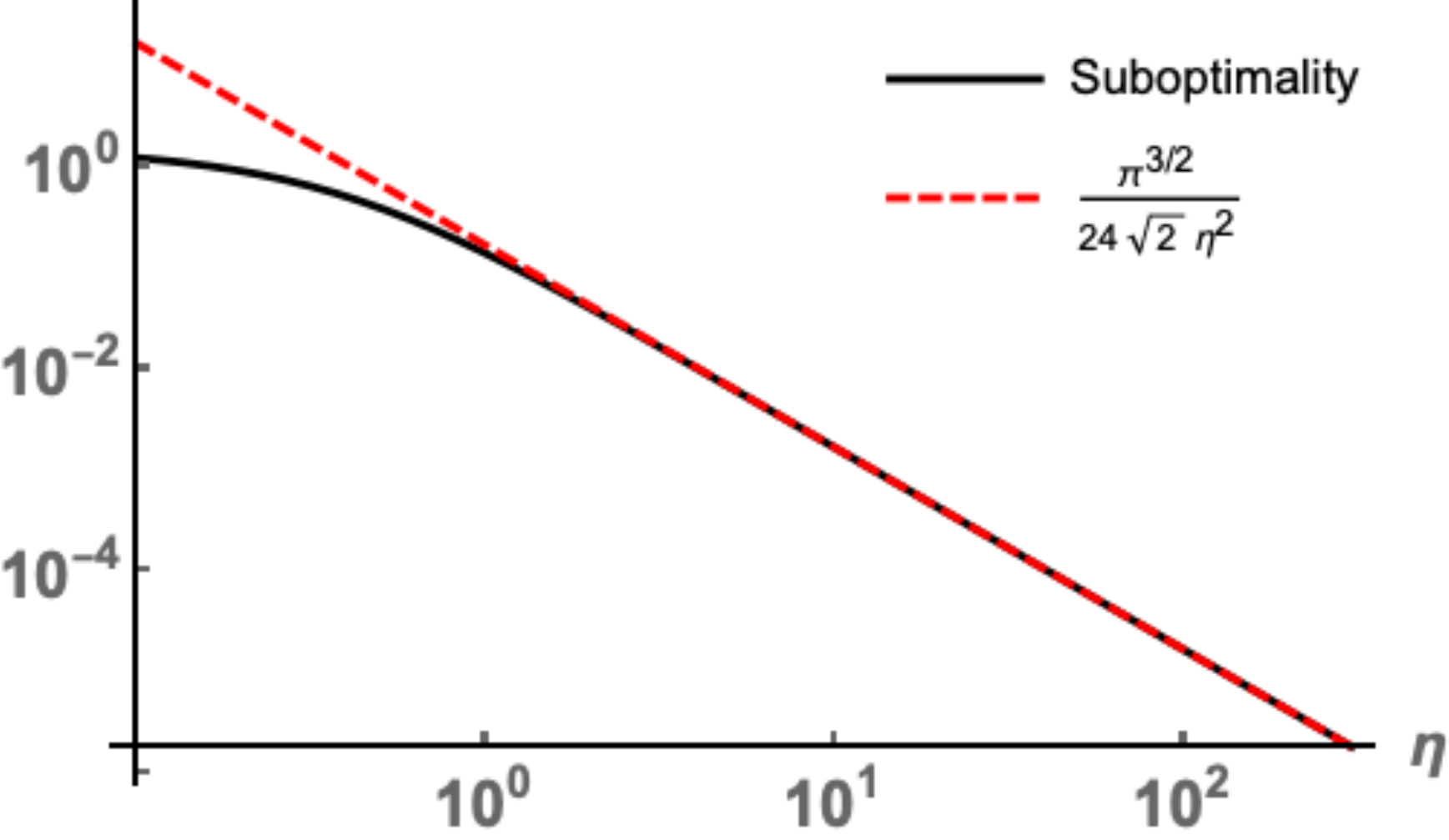}
		\caption{$\mu$ is standard Gaussian distribution, i.e., has density $\mu(x) = e^{-x^2/2}/\sqrt{2\pi}$.}
		\label{sttb:gaussian}
	\end{subfigure}
	\hfill
	\begin{subfigure}{0.45\linewidth} \centering
		\includegraphics[scale=0.38]{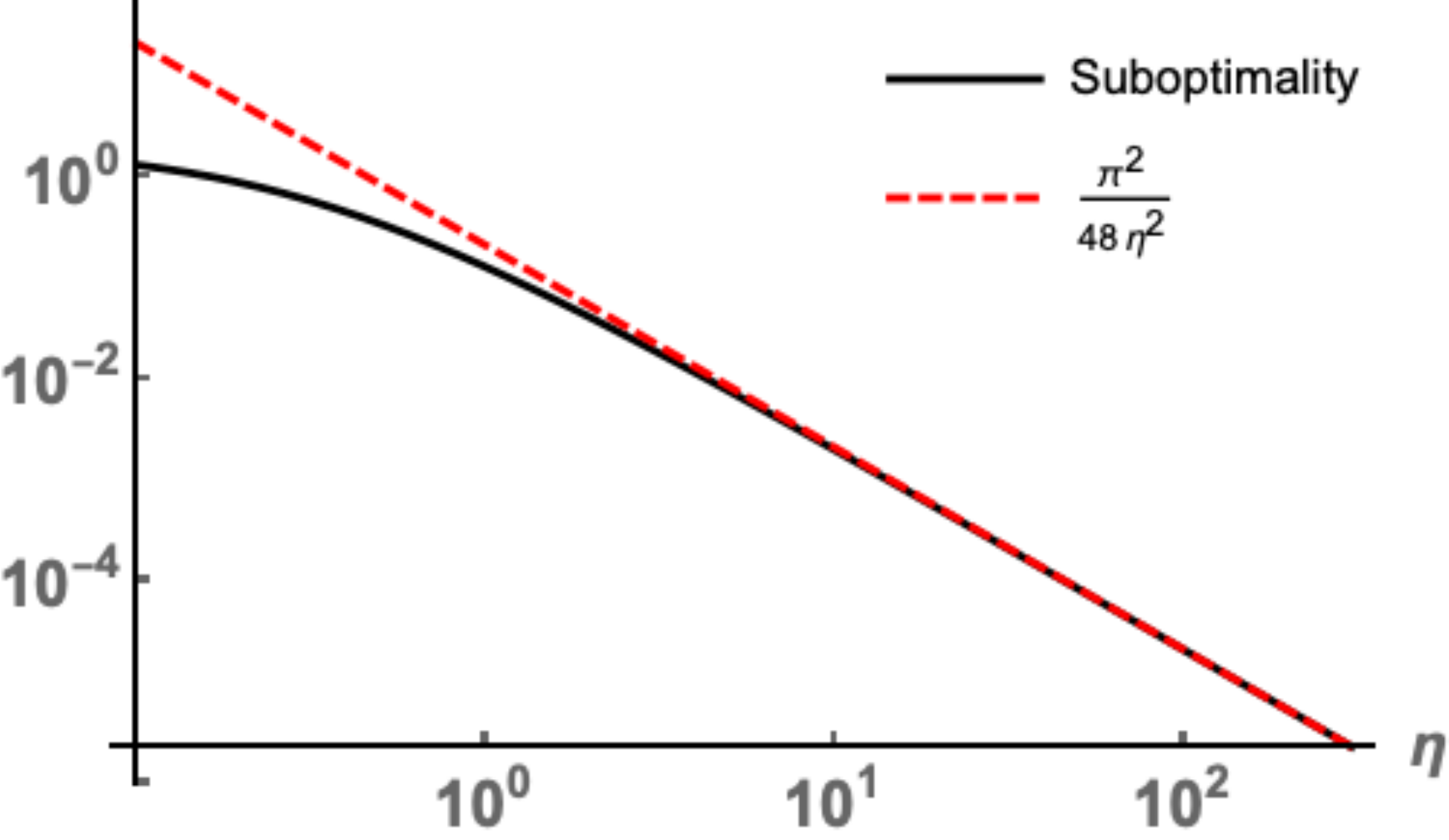}
		\caption{$\mu$ is standard Laplacian distribution, i.e., has density $\mu(x) = e^{-|x|}/2$.}
		\label{sttb:laplacian}
	\end{subfigure}
	\caption{
		For two toy examples in one dimension, simulations show that the suboptimality scales quadratically in $\eta^{-1}$, and that the leading constant is an explicit function of the value of the density at $0$.
		Our main result, Theorem~\ref{thm:subopt_limit}, extends this to the general setting. The agreement between the predicted limiting value and the simulation is precise.
	}
	\label{fig:sttb}
\end{figure}

Our main theorem shows that this phenomenon is completely general: in any dimension, if $\nu$ is discrete and $\mu$ has sufficiently regular density with respect to the Lebesgue measure, then the suboptimality scales as $\eta^{-2}$, with leading constant given by the value of $\mu$'s density on the hyperplanes on which the optimal map changes value.

\begin{theorem}\label{thm:subopt_limit}
	Suppose $\mu$ and $\nu$ are Borel probability measures on $\R^d$ such that $\nu$ is finitely supported on $y_1, \ldots, y_n$, and $\mu$ is absolutely continuous and compactly supported, with positive, continuous density on the interior of its connected support. 
	Then
	\begin{equation}\label{eqn:suboptimality_limit}
	\E_{\pi_\eta}[\|x - y\|^2] = W_2^2(\mu, \nu) + \frac{\zeta(2)}{2\eta^2} \sum_{i < j} \frac{w_{ij}}{\|y_i - y_j\|} + o(\eta^{-2})\,,
\end{equation}
where $w_{ij}$ is the $(d-1)$-dimensional integral of $\mu(x)$ on $\overline{T^{-1}(y_i)} \cap \overline{T^{-1}(y_j)}$ for the optimal map $T$ transporting $\mu$ to $\nu$ (see~\cref{ssec:prelim:geo}), and where $\zeta(2) = \frac{\pi^2}{6}$.
\end{theorem}
See \cref{sec:quadratic_sub_opt} for a precise statement and proof of this result.
The assumption that $\mu$ is compactly supported is mostly for convenience and can be substantially weakened; see \cref{continuous} and \cref{compact}.
By contrast, the continuity and positivity of $\mu$ are essential: in the absence of these assumptions, the convergence rate is no faster than $O(\eta^{-1})$ in general.

As an intermediate result, we also obtain an exact second-order expression for the cost with the entropic term. In what follows, we write $H(\nu) = -\sum_{i=1}^n \nu_i \log \nu_i$ to denote the Shannon entry of a discrete distribution $\nu$ with weights $\nu_1, \dots, \nu_n$ on its atoms.
\begin{theorem}\label{thm:cost_limit}
Suppose $\mu, \nu$ are as in Theorem~\ref{thm:subopt_limit}.
Then
\begin{equation}\label{cost_entropic_difference}
\E_{\pi_\eta}[\|x - y\|^2] + \frac 1 \eta \KL{\pi_\eta}{\mu \otimes \nu}  = W_2^2(\mu, \nu) + \frac{1}{\eta} H(\nu) - \frac{\zeta(2)}{2 \eta^2} \sum_{i< j} \frac{w_{ij}}{\|y_i - y_j\|} + o(\eta^{-2})\,.
\end{equation}
\end{theorem}
It would be interesting to find a heuristic argument to relate~\eqref{cost_entropic_difference} to~\eqref{eqn:second-order}.
In any case, the fact that the right side is $O(\eta^{-1})$ rather than $O(\eta^{-1} \log \eta)$ is a manifestation of the fact that the unregularized optimal coupling $\pi^*$ has finite relative entropy with respect to the product measure $\mu \otimes \nu$~\cite{nutz-notes}.

Prior work on the asymptotics of entropically regularized optimal transport has exploited a dynamical formulation~\cite{GenLeoRip17, GigTam20, CheGeoPav16} analogous to the celebrated Benamou--Brenier formula from the theory of unregularized optimal transport~\cite{BenBre00}.
However, to our knowledge, there is no rigorous formulation of such a principle for the semi-discrete setting.
We therefore take a different approach, similar in spirit to the one employed in the analysis of the discrete problem~\cite{cominetti1994asymptotic}, which focuses on the convex dual of~\eqref{eqn:eot}.
However, our proof techniques depart substantially from those available in the discrete case, where finite-dimensional considerations make the analysis of the dual problem more tractable.

Our main technical result, which is of possible independent interest, gives first-order asymptotics for the convergence of solutions of the convex dual of~\eqref{eqn:eot} to solutions of the dual of~\eqref{eqn:wass}, showing that this convergence happens faster than $\eta^{-1}$.

\begin{theorem}\label{thm:first_order_dual_convergence}
	Suppose $\mu, \nu$ are as in Theorem~\ref{thm:subopt_limit}.
	Let $(f_\eta, g_\eta)$ and $(f^*, g^*)$ solve the dual problems to \eqref{eqn:eot} and \eqref{eqn:wass}, respectively with appropriate normalization constraints.
	(See Definitions~\ref{defn:unreg_potential_convention} and \ref{defn:reg_potential_convention}.)
	Then
	\begin{align*}
		\eta(f_\eta - f^*) & \to 0 \\
		\eta(g_{\eta}  - g^*) &\to 0
	\end{align*}
	pointwise, with the latter convergence uniform.
\end{theorem}

\subsection{Related work}
The study of optimal transport dates back to the fundamental contributions of Monge in the 18th century~\cite{monge1781} and Kantorovich in the 20th~\cite{kantorovich1960mathematical}.
Later in the 20th century, significant progress was
made on the qualitative nature of
optimal transport solutions, with many independent
discoveries of a fundamental characterization of optimal transport solutions (Theorem~\ref{thm:fundamental_thm_ot}) ~\cite{brenier1987decomposition,knott1984optimal,cullen1984extended,cuesta1989notes,ruschendorf1990characterization}.
Around the turn of the 21st century, it was recognized that
optimal transport gives a deep geometric perspective on the space of probability distributions~\cite{mccann1997convexity, otto2001geometry}.
This discovery led to new functional inequalities, stable notions of curvature for metric measure spaces, and especially new means of analyzing difficult PDEs~\cite{otto2000generalization,sturm2006geometry, lott2009ricci, desvillettes2001trend,desvillettes2005trend}.

In parallel to these theoretical developments, major effort was devoted to practical
algorithms for computing optimal transport maps, particularly
in the discrete-discrete case.
Standard linear programming methods
work quite effectively when the supports of each distribution
are discrete with up to several thousand atoms~\cite{Dong2020,cuturi2013sinkhorn}.
However, for larger datasets linear programming methods
become prohibitively slow, and approximations are required.
The entropic regularization approach is the most popular approximation,
first considered algorithmically by Sinkhorn~\cite{sinkhorn1964relationship}
and Sinkhorn and Knopp~\cite{sinkhorn1967concerning} in the 1960s.
These works gave fast algorithms based off iterative matrix scaling
for computing the approximate optimal coupling.
Cuturi introduced this work to the machine learning community in 2013~\cite{cuturi2013sinkhorn},
which led to an explosion of interest in optimal transport for applications~\cite{computational_optimal_transport}.
Subsequently, the entropic penalty has been applied to variants of the optimal transport problem, where it also leads to fast and practical algorithms~\cite{ChiPeySch18,BenIjzRuk20,altschuler2020approximating,benamou2015iterative}.

Apart from its algorithmic implications, the entropic penalty has an interesting probabilistic interpretation dating back to Schr\"odinger.
In Schr\"odinger's original motivation, \eqref{eqn:eot} represents a formalization of the following \emph{hot gas experiment}.
Consider a collection of particles evolving according to Brownian motion, and suppose their initial and final distribution approximately coincide with the measures $\mu$ and $\nu$, respectively.
Schro\"dinger asked for a description of the ``most likely paths'' of each particle.
The entropically regularized optimal transport problem gives a way of making mathematical sense of this problem: the path measure governing the evolution of the particles can be obtained by convolving the optimal coupling $\pi_\eta$ given by the solution to~\eqref{eqn:eot} with a Brownian bridge~\cite{Fol88}.
This interpretation has led to a fruitful line of work understanding \eqref{eqn:eot} through the lens of large-deviations principles, which also has helped to clarify the nature of the convergence of~\eqref{eqn:eot} to~\eqref{eqn:wass} as $\eta \to \infty$~\cite{Leo12}.

Obtaining an asymptotic expansion of the cost $\E_{\pi_\eta}[\|x - y\|^2]$ or the entropic cost $\E_{\pi_\eta}[\|x - y\|^2] + \frac 1 \eta \KL{\pi}{\mu \otimes \nu}$ in the $\eta \to \infty$ limit is the subject of a great deal of recent interest.
In the discrete-discrete case, this question was first investigated in the broader context of entropically regularized linear programs by Cominetti and San Mart\'in~\cite{cominetti1994asymptotic}, who showed that the suboptimality converges to zero exponentially fast as $\eta \to \infty$.

In the continuous-continuous case, asymptotics have been computed to second order for the entropic cost, under regularity assumptions (see~\cite{ConTam21} and references therein).
To our knowledge, however, no general asymptotics for the suboptimality (without the entropic term) are known, but examples---such as the Gaussian case mentioned above---show that the rate $\Theta(\eta^{-1})$ is typical.

Recently, Bernton et al.~\cite{Bernton2021} developed a structural characterization of $\pi_\eta$ which allows them to establish a large-deviations principle for the convergence of $\pi_\eta$ to $\pi^*$, but they do not extract asymptotics for the cost.
Our results in Section~\ref{sec:prelim} develop a similar structural characterization for semi-discrete couplings by a direct argument.

The semi-discrete case, which is the central focus of this work, is important both for theoretical and practical reasons.
For instance, it reflects the practical situation of the statistician who has access to an empirical distribution $\nu = \frac 1n \sum_{i=1}^n \delta_{X_i}$ of samples from  an unknown measure, and wishes to compare these samples to an absolutely continuous reference measure $\mu$.
From a theoretical perspective, the semi-discrete setting is closely connected to the \emph{optimal quantization} problem~\cite{DerSchSch13, GraLus00, Pol82}, which seeks the best approximation of an absolutely continuous measure by a measure with finite support.
The study of the structure of optimal couplings for semi-discrete problems has a long history in computational geometry, where such couplings are known as ``power diagrams''~\cite{aurenhammer1987power,aurenhammer1998minkowski}.
We draw extensively on the properties of such diagrams in our geometrical results of Section~\ref{sec:prelim}.

\subsection{Organization of the remainder of the paper}
In \cref{sec:prelim}, we formalize several important definitions and establish some basic geometrical results on the structure of the optimal regularized and unregularized couplings. To illustrate our ideas, in \cref{sec:sttb} we develop the one-dimensional example mentioned above, and give a preview of the argument that will follow in the general case.
\Cref{sec:dual} contains the proof of our main technical result, \cref{thm:first_order_dual_convergence}, which is at the heart of our arguments.
In \cref{sec:quadratic_sub_opt}, we apply this convergence result to prove \cref{thm:subopt_limit}.
Finally, \cref{sec:appendix} contains necessary background information on the dilogarithm and zeta functions, as well as several intermediate integration lemmas needed for the proofs of our main theorems.
It also contains the proofs of two technical results from \cref{sec:prelim}.

\section{Background on semi-discrete OT and Sinkhorn problems}\label{sec:prelim}

In this section we recall relevant background on semi-discrete OT and Sinkhorn problems,
as well as provide several useful propositions and intuitions for the work that comes.
For further background we refer the reader to the standard textbooks~\cite{computational_optimal_transport,villani2008optimal},
as well as to the detailed treatment of the semi-discrete setting in~\cite[Section 4]{merigot2020optimal}.

\subsection{Semi-discrete optimal transport}
The foundational observation in optimal transport theory declares the existence, uniqueness,
and structure of the optimal coupling in the transport problem.
\begin{theorem}
	\label{thm:fundamental_thm_ot} Suppose $\mu, \nu$
	are probability measures with finite second moment. Then there is an optimal coupling
	$\pi^* \in \Pi(\mu, \nu)$ such that
	$$
	W_2^2(\mu, \nu) = \E_{\pi^*}[\|x - y\|^2].
	$$  Moreover, we have the following form of strong duality:
	\begin{equation}\label{eqn:w2_duality}
		W_2^2(\mu, \nu) = \sup_{(f, g) \in L^1(\mu) \times L^1(\nu)
		\colon f + g \leq \|x - y\|^2} \E_{\mu}
		\left[f\right]
		+  \E_{\nu}\left[g\right].
	\end{equation}
	If $\mu$ has a density with respect to the Lebesgue measure, then in fact
	there is a unique optimal
	$\pi^*$, it is supported on the graph of a function
	$T \colon \R^d \to \R^d$, and $T$ is the gradient of a (proper, lower semi-continuous)
	convex function.
	We shall usually write $T = T_{\mu \to \nu} = \nabla \phi_{\mu \to \nu}$.
	In this case the supremum in the dual problem~\eqref{eqn:w2_duality}
	is attained by
	$$
	(f, g) = (\|x\|^2 - 2\phi_{\mu \to \nu}, \|y\|^2 - 2 \phi_{\mu \to \nu}^c)
	$$ where we are using the Legendre conjugate
	$$
	\phi^c_{\mu \to \nu} (y) := \sup_x \langle x , y \rangle - \phi_{\mu \to \nu}(x).
	$$

\end{theorem}

The optimal $f$ and $g$ are typically not unique. However, the following assumptions guarantee that, up to an additive shift, $f$ and $g$ are unique $\mu$ (respectively, $\nu$) almost surely~\cite{BarGonLou21,Bernton2021}.

\begin{assume}\label{assume}
The measure $\nu$ is finitely supported and $\mu$ is absolutely continuous with finite second moment.
The interior of the support of $\mu$ is connected, the boundary of the support has zero Lebesgue measure, and $\mu$ has positive density on the interior of its support.
\end{assume}

Under Assumption~\ref{assume}, we can therefore uniquely identify a pair of optimal dual solutions.
\begin{defn}[Optimal unregularized potentials]\label{defn:unreg_potential_convention}
We denote by $(f^*, g^*)$ optimal solutions to~\eqref{eqn:w2_duality} subject to the additional normalization constraint that $\E_{\nu}[g^*] = 0$.
\end{defn}

Using Theorem~\ref{thm:fundamental_thm_ot}, we can completely characterize the optimal transport maps
in the semi-discrete case. In what follows, we identify $\mu$ with its Lebesgue density $\mu(\cdot)$, and write $\{y_i\}_{i = 1}^n$ for the support of $\nu$.

\begin{theorem}[\cite{aurenhammer1998minkowski}]\label{thm:dual-opt}
Adopt Assumption~\ref{assume}.
Then, $\mu$-almost surely,
	$$
	T_{\mu \to \nu}(x)  = \argmin_{y_i \in \supp(\nu)}(\|x - y_i\|^2 - g^*(y_i)).
	$$
\end{theorem}
\begin{proof} For ease of notation, write $\phi := \phi_{\mu\to \nu}$.
	Since $\phi$ is convex and closed, we know that $\phi = (\phi^c)^c$, where $(\cdot)^c$ denotes
	Legendre conjugation. Therefore,
	$$
	\phi(x) = \max_{y_i} \langle  x, y_i \rangle - \phi^c(y_i).
	$$
	Since $\mu$ is absolutely continuous, there is a unique maximizer for $\mu$-almost every $x$, and if $y_i$ is the unique maximizer for such an $x$, then $\nabla \phi(x) = y_i$, and 
	$$
	\|x - y_i\|^2 - \|y_i\|^2 + 2 \phi^c(y_i) < \|x - y_j \|^2 - \|y_j\|^2 + 2\phi^c(y_j) \quad \forall j \neq i\,.
	$$ 
	Therefore we have shown that $\mu$-almost everywhere,
	$$
	T(x) = \argmin_{y_i}( \|x - y_i\|^2 - (\|y_i\|^2 - 2\phi^c(y_i)) ).
	$$ This yields the result by the characterization in 
	Theorem~\ref{thm:fundamental_thm_ot}.
\end{proof}
In view of this result, the next definition is natural.

\begin{defn}[\cite{aurenhammer1987power}]\label{defn:power_cell} We define the \emph{power cells} with respect to the optimal dual potential $g^*$ by
	$$
	S_i := \{x \in \R^d \colon \forall j \, \, \|x - y_i\|^2 - g^*(y_i) \leq \|x - y_j \|^2 - g^*(y_j)\}, \quad i = 1, 
	\ldots, n.
	$$\end{defn}

\begin{center}
	\begin{figure}
		\centering
		\includegraphics[scale=.9]{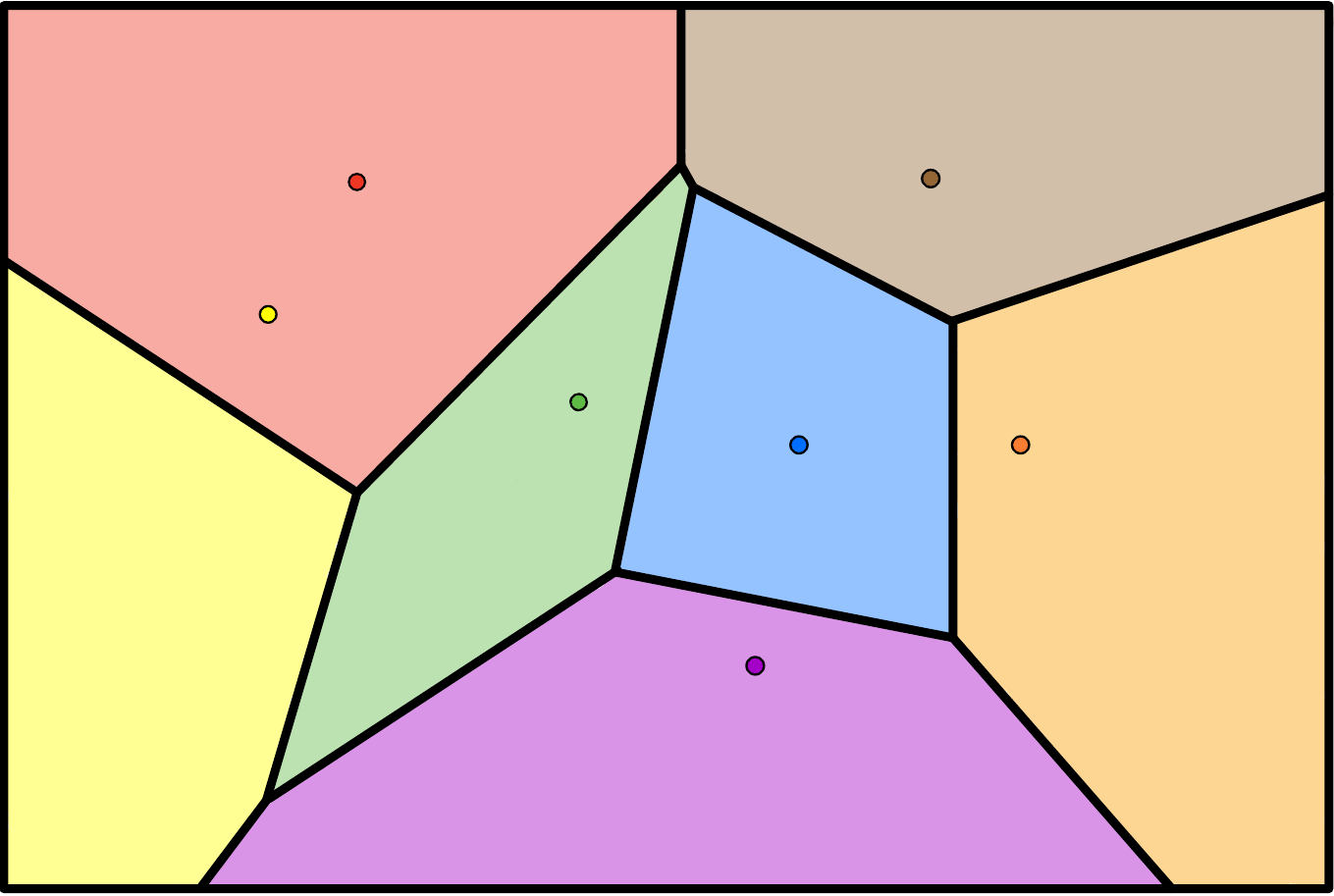}
		\caption{Illustration of a power cell diagram, or equivalently the optimal coupling for a semi-discrete OT problem. }
		\label{fig:power_cell}
	\end{figure}
\end{center}

The significance of the power cells $S_i$ is that they are precisely the pull-back of $y_i$ under 
$T_{\mu \to \nu}$:
$$
S_i = T^{-1}_{\mu \to \nu}(y_i).
$$
The power cells for $\pi^*$ form a convex polyhedral partition
of $\R^d$.
In Figure~\ref{fig:power_cell} we show an example of an optimal
mapping between a measure on the larger rectangle and a finitely
supported measure. Note that a point $y_i$ in the support of $\nu$ can lie in the power cell $S_j$ corresponding to a different point $y_j \neq y_i$. For example, this occurs if $\mu$ is supported on $(-\infty ,-2]$ and $\nu = (1/2) \delta_{-1} + (1/2) \delta_1$.

\subsection{Semi-discrete entropic optimal transport}

In this subsection, we discuss the entropy regularized version of the semi-discrete
optimal transport problem.
Denote by $\rho$ the counting measure on the support of $\nu$.
We first note that for any $\pi \in \Pi(\mu, \nu)$, we have
\begin{equation}\label{ent_shift}
\KL{\pi}{\mu \otimes \nu} =\KL{\pi}{\mu \otimes \rho} + H(\nu)\,.
\end{equation}
The regularized optimal transport problem~\eqref{eqn:eot} is therefore equivalent to
\begin{equation}\label{eq:eot_counting}
\inf_{\pi \in \Pi(\mu, \nu)} \E_{\pi}[\|x - y\|^2] + \frac{1}{\eta}\KL{\pi}{\mu \otimes \rho}\,.
\end{equation}
The benefit of the formulation~\eqref{eq:eot_counting} is that under Assumption~\ref{assume},
\begin{equation*}
\KL{\pi^*}{\mu \otimes \rho} = 0\,,
\end{equation*}
which leads to a simplification in some of the formulas appearing in what follows.

Csisz\'ar's theory of ``I-projection''~\cite{Csi75} implies that as long as $\mu$ and $\nu$ have finite second moment, the value of~\eqref{eq:eot_counting} equals the value of the dual problem
\begin{align}\label{eqn:sd_sink_dual}
	\sup_{(f,g) \in L^1(\mu) \times L^1(\nu)}
	\E_{\mu}[f] + \E_{\nu}[g] - \frac{1}{\eta}\sum_{j = 1}^n
	\int_{\R^d} e^{-\eta(\|x - y_j\|^2 - f(x) - g(y_j))} \mu(x)dx + \frac 1 \eta\,.
\end{align}
Moreover, the optimal solution to~\eqref{eq:eot_counting} satisfies
\begin{equation}\label{eq:dual_opt}
\frac{d\pi_{\eta}}{d (\mu \otimes \rho)}(x, y) = e^{-\eta(\|x - y\|^2 - f_{\eta}(x) - g_{\eta}(y))}\,,
\end{equation}
where $f_\eta$ and $g_\eta$ solve~\eqref{eqn:sd_sink_dual}.

The strict convexity of~\eqref{eqn:sd_sink_dual} implies that $f_\eta$ and $g_\eta$ are unique up to an additive shift; as above, we therefore fix a unique optimal pair by adding an additional constraint.

\begin{defn}[Optimal regularized potentials] \label{defn:reg_potential_convention}
We denote by $(f_{\eta}, g_{\eta})$ solutions to~\eqref{eqn:sd_sink_dual}, subject to the additional normalization constraint $\E_{\nu}[g_{\eta}] = 0$.
\end{defn}

\subsection{Useful geometric notions}\label{ssec:prelim:geo}
The power cell decomposition of Definition~\ref{defn:power_cell} gives us a useful way to separate the subproblems arising in our proof into individual problems over the cells $S_i$.
In the service of analyzing these problems, we will focus on the distance of a point in $x \in S_i$, from each of the hyperplanes defining $S_i$.
We call these quantities the \emph{slacks}, in reference to the fact that they represent the slack in the dual feasibility constraints in~\eqref{eqn:w2_duality}.

\begin{defn}[Slack]\label{defn:slack}
    Let $i,j \in [n]$. The $j$-th slack at point $x \in S_i$ is
\begin{equation}\label{eqn:deltaij}
	\Delta_{ij}(x):= \| x - y_j\|^2 - f^*(x) - g^*(y_j).
\end{equation}
\end{defn}

We establish several basic properties of this slack operator.

\begin{lemma}[Properties of slack]\label{lem:slack_properties} 
    For $i,j \in [n]$ and $x \in S_i$,
    \begin{itemize}
        \item \underline{Nonnegativity.} $\Delta_{ij}(x) \geq 0$, with strict inequality $\mu$-almost everywhere if $i \neq j$.
        \item \underline{Diagonals vanish.} $\Delta_{ij}(x) = 0$ if $i=j$.
        \item \underline{Expression via $g^*$.}
		$\Delta_{ij}(x) = 2 \langle x, y_i - y_j \rangle - \|y_i\|^2 + \|y_j\|^2 - g^*(y_j) + g^*(y_i)$.
	  \end{itemize}
\end{lemma}
\begin{proof}
    Nonnegativity follows by feasibility of $(f^*,g^*)$ for the dual OT problem~\eqref{eqn:w2_duality}, with strict inequality following from the fact that $\|x - y_i\|^2 - g^*(y_i) < \|x - y_j\|^2 - g^*(y_j)$ in the interior of $S_i$.
     The vanishing $\Delta_{ii} \equiv 0$ follows from the fact that $\|x - y\|^2 - f^*(x) - g^*(y) = 0$ $\pi^*$-almost surely, by strong duality.
     For the final item, observe that
    \[
    	\Delta_{ij}(x) = \|x - y_j\|^2 - f^*(x) - g^*(y_j)
    	= \|x - y_j\|^2 - \|x - y_i\|^2 + g^*(y_i) - g^*(y_j)
    \]
	where the second step is because $\|x-y_i\|^2 = f^*(x) + g_i^*$ by 
	the previous item $\Delta_{ii}(x)$ = 0.
	Now expand the square.
\end{proof}

Our second main assumption on the measure $\mu$ relates to the regularity of the density along level sets defined by the slacks.
We require several definitions.
For $i \neq j$ and $a \geq 0$, set
\begin{align*}
S_{ij}(a) & := \{ x \in \R^d : \|x-y_i\|^2 - g_i^* \leq \|x-y_k\|^2 - g_k^* - a \mathbf{1}_{k \neq i,j}, \;\; \forall k \in [n]\} \\
& = \{x \in S_i: \Delta_{ik}(x) \geq a, \;\; \forall k \neq i, j\}\,.
\end{align*}
When $a = 0$, $S_{ij}(0) = S_i$.
For $a > 0$, $S_{ij}(a)$ is the subset obtained from $S_i$ by pushing in the hyperplanes separating $S_i$ from all neighboring cells other than $S_j$.
Likewise, for $t \geq 0$, we let $H_{ij}(t;a) = \{x \in S_{ij}(a): \Delta_{ij}(x) = t\}$ be the intersection of this set with a hyperplane parallel to the boundary between $S_i$ and $S_j$. See Figure~\ref{fig:hiij} for an illustration.

Since $\mathbbold{1}[x \in S_{ij}(a)] \mu(x)$ is in $L^1(\R^d)$, we can define
\begin{equation}\label{eqn:h_def}
h_{ij}(t; a) := \int_{H_{ij}(t;a)} \mu(x) d \mathcal H_{d-1}(x) \in L^1(\R)
\end{equation}
where $\mathcal H_{d-1}$ denotes the $(d - 1)$-dimensional Hausdorff measure on $H_{ij}(t;a)$.
When $a = 0$, we abbreviate $H_{ij}(t; a)$ and $h_{ij}(t; a)$ by $H_{ij}(t)$ and $h_{ij}(t)$, respectively.

The benefit of this definition is that it gives us a convenient way to integrate functions that depend only on the slacks; indeed, the coarea formula implies that for any nonnegative $\phi: \R \to \R$,
\begin{equation*}
\int_{S_{ij}(a)} \phi(\Delta_{ij}(x)) \mu(x) dx = \frac{1}{2 \|y_i - y_j\|} \int_0^\infty \phi(t) h_{ij}(t; a) dt\,.
\end{equation*}

\begin{center}
	\begin{figure}
		\centering
		\includegraphics[scale=0.9]{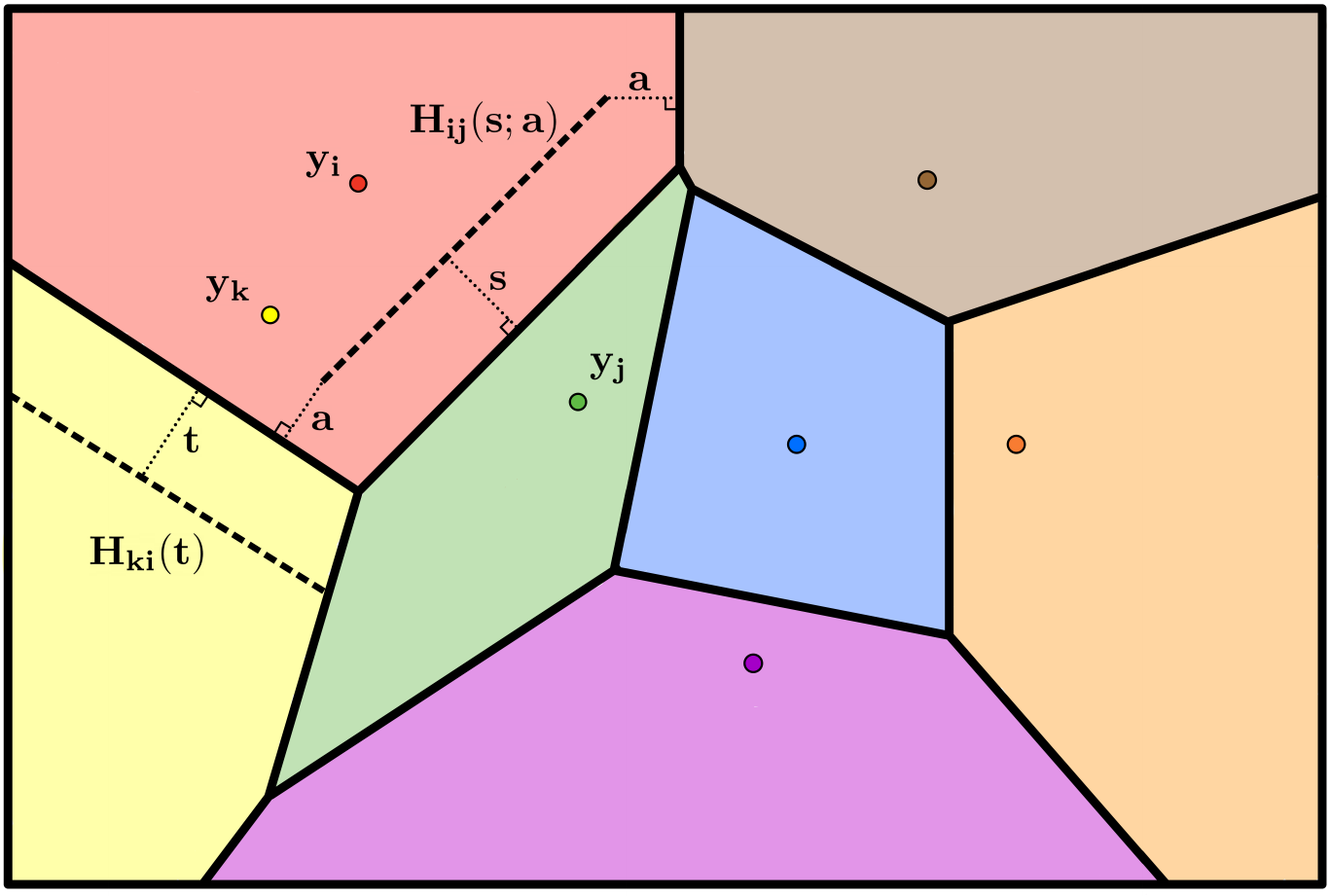}
		\caption{Power cell diagram with our $H_{ki}(t)$, $H_{ij}(s; a)$ notation depicted.
		}
		\label{fig:hiij}
	\end{figure}
\end{center}

We require the following crucial condition on the measure $\mu$.
\begin{assume}\label{continuous}
For all $i \neq j$ and $a \geq 0$ sufficiently small, the functions $t \mapsto h_{ij}(t; a)$ and $a \mapsto h_{ij}(0; a)$ are continuous at $0$.
\end{assume}

Assumption~\ref{continuous} is a strong requirement on the regularity of $\mu$ along hyperplanes, and it is essential for our results.
As alluded to in the statement of \cref{thm:subopt_limit}, it is possible to verify \cref{continuous} under easy conditions on $\mu$.
Say that $\mu$ is \emph{dominated along hyperplanes} if for any affine hyperplane $H$ orthogonal to a vector $v$ there exists a nonnegative $\psi: \R^{d-1} \to \R$, integrable with respect to the Lebesgue measure, and an affine isometry $P: H \to \R^{d-1}$ such that
\begin{equation*}
\mu(x + t v) \leq \psi(P x) \quad \forall t \in \R, x \in H\,.
\end{equation*}
If $\mu$ is pointwise bounded and compactly supported, then it is dominated along hyperplanes; however, some non-compactly supported measures, such as the standard Gaussian measure on $\R^d$ also enjoy this property.

\begin{prop}\label{compact}
If $\mu$ is continuous and dominated along hyperplanes, then \cref{continuous} holds.
\end{prop}

Finally, we record a simple consequence of the connectedness of the support of $\mu$, which we will rely on extensively in \cref{sec:dual}.
\begin{lemma}\label{lem:connected}
Under Assumption~\ref{assume}, we have $h_{ij}(0) = h_{ji}(0)$ for all $i \neq j$, and the graph on $[n]$ with edge set $\{(i, j): h_{ij}(0) > 0\}$ is connected.
\end{lemma}

The proofs of \cref{compact,lem:connected} appear in \cref{sec:appendix}.

\section{Case study: symmetric one-dimensional measures}\label{sec:sttb}
In order to provide intuition for our main result, we consider here a toy example which, despite its simplicity, illustrates many of the key underlying phenomena. Specifically, in this section we explicitly compute the suboptimality in the case where $\mu$ has a symmetric density on $\R$ and $\nu$ is the discrete distribution $\nu = (1/2) \delta_{-1} + (1/2)\delta_1$. The symmetry of both distributions around $0$ allows us to compute closed-form expressions for $\pi^*$ and $\pi_{\eta}$, and hence also for the suboptimality. These closed-form expressions hold for any $\eta > 0$ and facilitate understanding our assumptions and main techniques.

\paragraph*{Unregularized optimal transport plan $\pi^*$.} By symmetry of $\mu$, the optimal coupling $\pi^*$ is supported on the graph a function that sends $x \in \supp(\mu)$ to $\sgn(x)$. That is, 
\begin{align*}
	\pi^*(x, y) = \mathbbold{1}[y = \sgn(x)]\cdot \mu(x).
\end{align*}

\paragraph*{Regularized optimal transport plan $\pi_{\eta}$.} Let us compute the dual
potentials $f_{\eta}, g_{\eta}$ from Definition~\ref{defn:reg_potential_convention}.
Symmetry of the distributions around $0$ implies
$$
\pi_{\eta}(x, y) = \pi_{\eta}(-x, -y).
$$ 
Using~\eqref{eq:dual_opt} and solving, this means $f_{\eta}(x) - f_{\eta}(-x) = g_{\eta}(-y) - g_{\eta}(y)$
for all $x \in \supp(\mu)$ and $y \in \supp(\nu)$. Replacing $x$ with $-x$, we see that
both $f_{\eta}$ and $g_{\eta}$ must be even functions. By our 
convention in Definition~\ref{defn:reg_potential_convention},
it follows that $g_{\eta}(1) = g_{\eta}(-1) = 0$.

\par We can now solve for $f_{\eta}$ using the marginal constraint $\mu(x) = \pi_{\eta}(x, 1) + \pi_{\eta}(x, -1)$. Plugging in the optimality conditions~\eqref{eq:dual_opt} for $\pi_{\eta}$ and simplifying implies
$$
e^{\eta f_{\eta}(x)}
= \frac{1}{e^{-\eta(x - 1)^2} + e^{-\eta(x + 1)^2}}.
$$
Rearranging, we conclude that
\begin{equation}\label{eqn:sttb_pi_eta}
	\pi_{\eta}(x, y) = \frac{e^{-\eta(x - y)^2 }}{e^{-\eta(x - 1)^2 }
		+ e^{-\eta(x + 1)^2 }}\mu(x) = \frac{\mu(x)}{e^{2\eta x (1 - y)} + e^{-2\eta x(1 + y)}}.
\end{equation}
See Figure~\ref{sttb:cond} for an intuitive interpretation of $\pi_{\eta}$ as a smoothed version of $\pi^*$.

\begin{figure}
	\centering
	\includegraphics[width=0.4\textwidth]{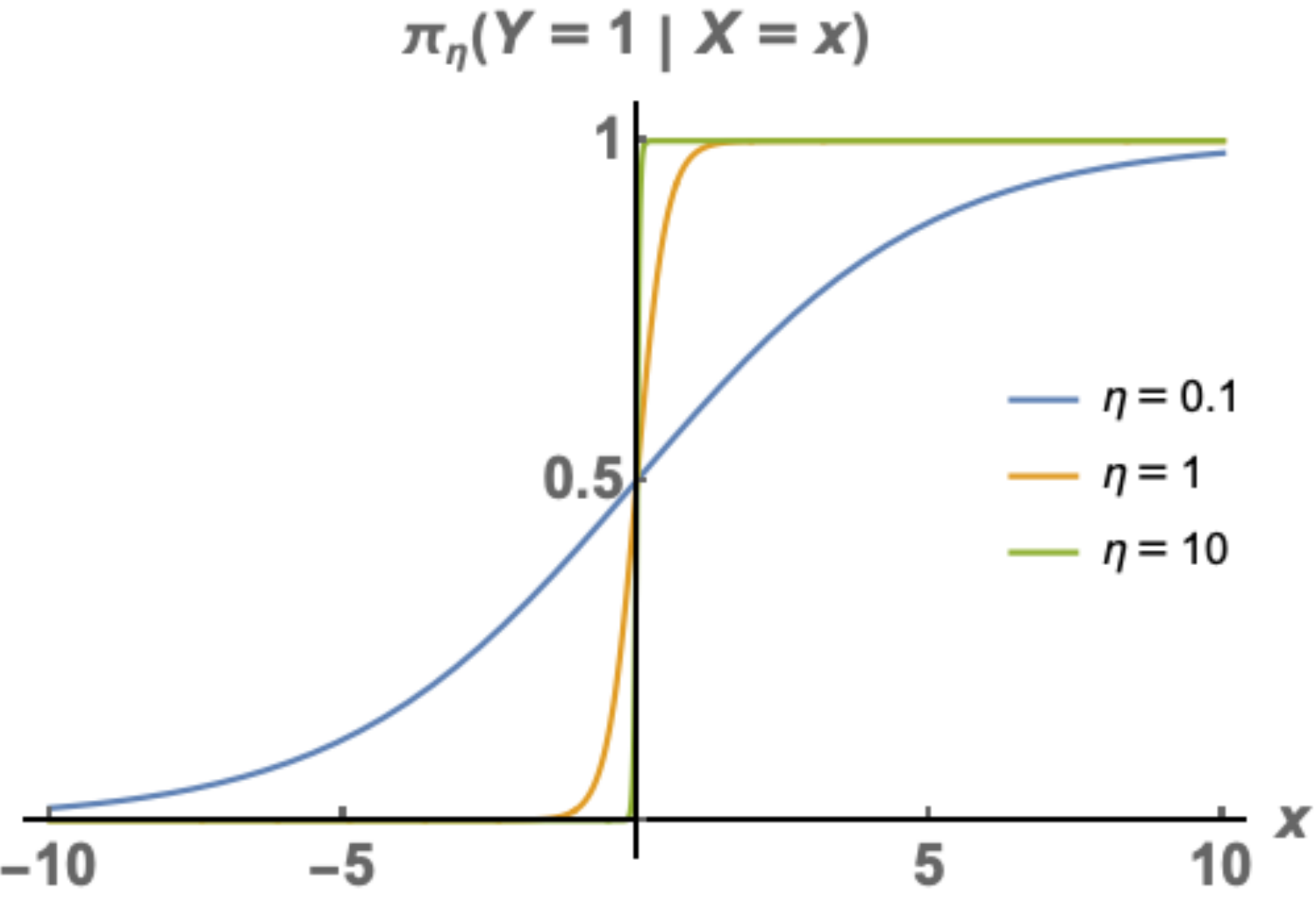}	\caption{For the toy example in Section \ref{sec:sttb}, the conditional distribution 
	$\pi_{\eta}(Y=1|X=x)$ of the regularized plan $\pi_{\eta}$ is the sigmoid function $1/(1 + e^{-4\eta x})$ by~\eqref{eqn:sttb_pi_eta}. As $\eta \to \infty$, this converges to the conditional distribution $\pi^*(Y=1|X=x) = \mathbbold{1}[\sign(x) = 1]$ of the unregularized plan $\pi^*$. The convergence is exponential in $\eta$ at any $x \neq 0$. There is a symmetric region around the origin of width $\Theta(1/\eta)$ on which $\pi_{\eta}(Y=1|X=x)$ is bounded away from $0$ and $1$.
	}
	\label{sttb:cond}
\end{figure}

\paragraph*{Explicit evaluation of suboptimality.} By symmetry, marginal constraints, and the
formula~\eqref{eqn:sttb_pi_eta}, we find
\begin{align}
	\E_{\pi_{\eta}}[(x - y)^2] - \E_ { \pi^*}[(x - y)^2]
	&= 2\int_{0}^{\infty} \left((x - 1)^2 (\pi_{\eta}(x, 1) - 1)+ (x + 1)^2\pi_{\eta}(x, -1)
	\right) dx \nonumber\\
	&= 2\int_0^{\infty} ((x + 1)^2  - (x - 1)^2) \pi_{\eta}(x, -1)dx
	\nonumber  \\
	&= 8 \int_0^{\infty} \frac{x}{ 1 + e^{4\eta x}} \mu(x) dx. \label{eqn:sttb_sub_opt}
\end{align}

\noindent 
The dominant part of~\eqref{eqn:sttb_sub_opt} as $\eta \to \infty$ is at $x = 0$, and if $\mu$ is continuous it can be shown that it is valid to replace $\mu(x)$ by $\mu(0)$ to obtain
\begin{equation*}
\E_{\pi_{\eta}}[(x - y)^2] - \E_ { \pi^*}[(x - y)^2] \approx 8 \int_0^{\infty} \frac{x}{ 1 + e^{4\eta x}} \mu(0) dx = - \frac{\Li_2(-1) \mu(0)}{2 \eta^2} = \frac{\pi^2 \mu(0)}{24 \eta^2}\,.
\end{equation*}
Here, $\Li_2$ is the \emph{dilogarithm} function, which will play a central role in our argument.
More details about this function can be found in \cref{sec:appendix}. In particular, the above integral identity is by~\cref{log_ints}.

\paragraph*{Necessity of assumptions.} 
If $\mu$ fails to be continuous at zero, convergence to $0$ may be slower than quadratic.
Consider $\mu(x) = c_p |x|^{-p}$ on $[-1, 1]$ for $p < 1$ and normalizing constant $c_p = (1 - p)/2$. The analysis above holds unchanged up to Equation~\ref{eqn:sttb_sub_opt}. However, the following step, in which we approximated the integral by replacing $\mu(x)$ with $\mu(0)$, does not hold here since $\mu$ is not continuous at $0$. Specifically,
$$
\E_{\pi_{\eta}}[(x - y)^2] - \E_ { \pi^*}[(x - y)^2] = 8c_p \int_0^{1} \frac{x^{ 1- p}}{1 + e^{4\eta x}}dx 
= \frac{2c_p}{\eta^{2 - p}} \int_0^{4\eta} \frac{u^{1-p}}{1 + e^{u}}du 
= \Theta\left(\frac{1}{\eta^{2 - p}}\right).
$$ This shows that in fact any polynomial rate faster than $1/\eta$ 
is achievable when our assumptions are violated. Morever, taking $\mu$ supported away from $0$ shows that an exponential rate can be obtained when $\mu$ is not supported at the decision boundary.

\section{Convergence of dual potentials}\label{sec:dual}
In this section, we develop an asymptotic expansion for the solution $g_\eta$ of~\eqref{eqn:sd_sink_dual} around the optimal solution $g^*$ to the unregularized problem~\eqref{eqn:w2_duality}.
Recall that Assumption~\ref{assume} implies that $g^*$ is unique, and it is easy to see~\cite{Nutz2021} that under this assumption $g_\eta$ converges to $g^*$.
The main result of this section is a more precise result, showing that this convergence happens at the rate $o(\eta^{-1})$.

We prove the following.
\begin{theorem}\label{dual_convergence}
Under Assumptions~\ref{assume} and~\ref{continuous},
\begin{equation*}
\lim_{\eta \to \infty} \|\eta(g_\eta - g^*)\|_{\infty} = 0\,.
\end{equation*}
\end{theorem}
A consequence of \cref{dual_convergence} is that $\eta(f_\eta - f^*) \to 0$ pointwise, though we stress that this convergence is not uniform.

From the general theory of entropic optimal transport, this result is unexpected, and it reflects particular features of the semi-discrete setting.
For instance, when $\mu$ and $\nu$ are both discrete the quantities $\eta(g_\eta - g^*)$ and $\eta(f_\eta -f^*)$ both converge to positive limits in general.
Moreover, Assumption~\ref{assume} is essential: if $\mu$ is not positive on the interior of its support, it is possible for $\eta(g_\eta - g^*)$ to diverge.\footnote{This occurs, for instance, when $\mu$ decays to zero at different rates on opposite sides of one of the hyperplane boundaries $H_{ij}$.}

The proof of \cref{dual_convergence} also yields the following corollary on the difference between the Wasserstein distance and the entropic cost, which gives~\cref{thm:cost_limit}.
\begin{corollary}\label{cost_cor}
Under Assumptions~\ref{assume} and~\ref{continuous},
\begin{equation*}
\lim_{\eta \to \infty} \eta^2 \left(\E_{\pi^*}[\|x - y\|^2] -  (\E_{\pi_\eta}[\|x - y\|^2] + \frac 1 \eta \KL{\pi_\eta}{\mu \otimes \rho})\right) = \frac{\zeta(2)}{2} \sum_{i < j} \frac{h_{ij}(0)}{\|y_i - y_j\|}\,.
\end{equation*}
Equivalently,
\begin{equation}
\E_{\pi_\eta}[\|x - y\|^2] + \frac 1 \eta \KL{\pi_\eta}{\mu \otimes \nu}  = W_2^2(\mu, \nu) + \frac{1}{\eta} H(\nu) - \frac{\zeta(2)}{2\eta^2} \sum_{i < j} \frac{h_{ij}(0)}{\|y_i - y_j\|} + o(\eta^{-2})\,.
\end{equation}
\end{corollary}

To prove~\cref{dual_convergence}, we define the function
\begin{equation*}
d_\eta := \eta(g_\eta - g^*)\,.
\end{equation*}
We will show that $d_\eta$ is the unique solution to an auxiliary convex optimization problem whose solution gives the first-order difference between the Wasserstein distance $W_2^2(\mu, \nu)$ and the entropic cost $\E_{\pi_\eta}[\|x - y\|^2] + \frac 1 \eta \KL{\pi_\eta}{\mu \otimes \rho}$.
By showing that the zero function is an approximate optimizer of this auxiliary problem and establishing a form of strong convexity around $0$ in the limit, we obtain that $d_\eta \to 0$, proving the claim.

We begin by defining these auxiliary optimization problems.
\begin{prop}\label{prop:dual}
The function $d_\eta$ is the unique solution of
\begin{equation}\label{d_eq}
\min_{d \in L_1(\nu)\,:\, \E_\nu d = 0} \sum_{i=1}^n \int_{S_i} \log(1 + \sum_{j \neq i} e^{d(y_j) - d(y_i) - \eta \Delta_{ij}(x)}) \mu(x) dx\,.
\end{equation}
Moreover, if we denote by $\Phi(\eta)$ the value of~\eqref{d_eq}, then
\begin{equation}\label{eq:phi_value}
\Phi(\eta) = \eta  \left(\E_{\pi^*}[\|x - y\|^2] -  (\E_{\pi_\eta}[\|x - y\|^2] + \frac 1 \eta \KL{\pi_\eta}{\mu \otimes \rho})\right)\,,
\end{equation}
and $\pi_\eta$ satisfies
\begin{equation}\label{coupling_with_d}
\frac{d \pi_\eta}{d (\mu \otimes \rho)} (x, y_j) = \frac{e^{d_\eta(y_j) - \eta \Delta_{ij}(x)}}{\sum_{k} e^{d_\eta(y_k) - \eta \Delta_{ik}(x)}} \quad \quad \quad \forall x \in S_i, i \in [n]\,.
\end{equation}
\end{prop}
\begin{proof}
Recall that $f_\eta$ and $g_\eta$ are the unique solutions to~\eqref{eqn:sd_sink_dual} subject to the constraint $\E_\nu[g_\eta] = 0$, so they also uniquely solve
\begin{equation*}
\eta \cdot \min_{\substack{(f, g) \in L_1(\mu) \times L_1(\nu) \\ \E_\nu[g] = 0}} \E_{\mu}[f^*] + \E_\nu[g^*] - \E_\mu[f] - \E_\nu[g] + \frac 1 \eta \sum_{j=1}^n \int_{\R^d} e^{-\eta(\|x - y_j\|^2 - f(x) - g(y_j))} \mu(x) dx - \frac 1 \eta\,.
\end{equation*}
By duality, the optimal value of this program is exactly~\eqref{eq:phi_value}.
Decomposing the integrals over the cells $S_i$ and recalling~\eqref{eqn:deltaij}, we obtain that $f_\eta$ and $g_\eta$ are the unique solutions to
\begin{multline}
\min_{\substack{(f, g) \in L_1(\mu) \times L_1(\nu) \\ \E_\nu[g] = 0}} \sum_{i=1}^n \int_{S_i} \Big( \eta(f^*(x) - f(x)) + \eta(g^*(y_i) - g(y_i)) \\ + \sum_{j=1}^n e^{-\eta(\Delta_{ij}(x) + f^*(x) - f(x) + g^*(y_j) - g(y_j))}\Big) \mu(x) dx - 1\,.
\end{multline}

Reparametrizing in terms of $\delta_f = \eta(f - f^*)$ and $\delta_g = \eta(g - g^*)$ yields the equivalent representation
\begin{equation*}
\min_{\substack{(\delta_f, \delta_g) \in L_1(\mu) \times L_1(\nu) \\ \E_\nu[\delta_g] = 0}} \sum_{i=1}^n \int_{S_i} \Big(-1 -\delta_f(x) - \delta_g(y_i) + \sum_{j=1}^n e^{\delta_f(x) + \delta_g(y_j) -\eta\Delta_{ij}(x)}\Big) \mu(x) dx\,,
\end{equation*}
with optimal solutions $\eta(f_\eta - f^*)$ and $\eta(g_\eta - g^*)$.
Fixing $\delta_g$ and minimizing this expression with respect to $\delta_f$ yields that the optimal solutions $\delta_f$ and $\delta_g$ are related by
\begin{equation*}
\delta_f(x) = - \log\big(\sum_{j=1}^n e^{\delta_g(y_j) - \eta \Delta_{ij}(x)}\big)
\end{equation*}
for $\mu$-almost every $x \in S_i$.
Plugging in this expression gives
\begin{multline*}
\min_{\delta_g \in L_1(\nu)\,:\,\E_\nu[\delta_g] = 0} \sum_{i=1}^n \int_{S_i} \Big(\log\big(\sum_{j=1}^n e^{\delta_g(y_j) - \eta \Delta_{ij}(x)}\big) -\delta_g(y_i) \Big)\mu(x) dx \\
= \min_{\delta_g \in L_1(\nu): \E_\nu[\delta_g] = 0} \sum_{i=1}^n \int_{S_i} \log\big(1 + \sum_{j \neq i} e^{\delta_g(y_j) - \delta_g(y_i) - \eta \Delta_{ij}(x)}\big) \mu(x) dx\,.
\end{multline*}
Writing $d$ for $\delta_g$ yields~\eqref{d_eq}.

Finally, applying the same argument to~\eqref{eq:dual_opt} yields
\begin{align*}
\frac{d \pi_\eta}{d (\mu \otimes \rho)} (x, y_j) & = e^{-\eta(\|x - y_j\|^2 - f_{\eta}(x) - g_{\eta}(y_j)} \\
& = e^{\delta_f(x) + \delta_g(y_j) -\eta\Delta_{ij}(x)} \\
& = \frac{e^{d_\eta(y_j) - \eta \Delta_{ij}(x)}}{\sum_{k } e^{d_\eta(y_k) - \eta \Delta_{ik}(x)}}
\end{align*}
for all $x \in S_i$ and $i \in [n]$, as desired.
\end{proof}

To prove the theorem, we require two intermediate results.
First, we obtain an upper bound on $\Phi$ by comparing it to the value of~\eqref{d_eq} at $d = 0$.
Though crude, this comparison will turn out to be accurate to first order.

\begin{lemma}\label{lem:comparison_limit}
\begin{equation*}
\limsup_{\eta \to \infty} \eta \Phi(\eta)
\leq \frac{\zeta(2)}{4} \sum_{i \neq j} \frac{h_{ij}(0)}{\|y_i - y_j\|}\,.
\end{equation*}
\end{lemma}
\begin{proof}
Choose $d = 0$ in~\eqref{d_eq}.
The subadditivity of the function $\alpha \mapsto \log(1+\alpha)$ for $\alpha > 0$ and the optimality of $d(\eta)$ then imply
\begin{align*}
\Phi(\eta) & \leq \sum_{i=1}^n \int_{S_i} \log\left(1 + \sum_{j \neq i} e^{-\eta \Delta_{ij}(x)}\right) \mu(x) dx \\
& \leq \sum_{i=1}^n \sum_{j \neq i} \int_{S_i} \log(1+ e^{-\eta \Delta_{ij}(x)}) \mu(x) dx\,.
\end{align*}
Multiplying by $\eta$, taking the limit, and applying~\cref{log_ints_si} yields the claim.
\end{proof}

Next, we use \cref{lem:comparison_limit} to show that the solutions to~\eqref{d_eq} remain bounded.

\begin{prop}\label{dual_bounded}
Under \cref{assume}, $d_\eta$ is bounded as $\eta \to \infty$.
\end{prop}
\begin{proof}
The claim is obvious if $n = 1$, so assume $n \geq 2$.
Fix $(i, j)$ for which $h_{ij}(0) > 0$.
(Such a pair exists by \cref{lem:connected}.)
Then by \cref{prop:dual},
\begin{align*}
\eta \Phi(\eta) & = \eta \sum_{i=1}^n \int_{S_i} \log(1 + \sum_{j \neq i} e^{d_\eta(y_j) - d_\eta(y_i) - \eta \Delta_{ij}(x)}) \mu(x) dx \\
& \geq \eta \int_{S_i} \log(1 + e^{d_\eta(y_j) - d_\eta(y_i) - \eta \Delta_{ij}(x)}) \mu(x) dx
\end{align*}
To bound this integral, we require the following lemma, which we prove below.
\begin{lemma}\label{lem:log_subadditive}
For any $a \geq 0$ and $b \in [0,1]$,
\begin{equation}\label{eq:log_subadditive}
\log(1+ab) \geq \log(1+a)\log(1+b)\,.
\end{equation}
\end{lemma}

With this lemma in hand, we obtain
\begin{align*}
\eta \Phi(\eta)  & \geq \log(1+e^{d_\eta(y_j) - d_\eta(y_i)}) \cdot \eta \int_{S_i} \log(1+e^{-\eta \Delta_{ij}(x)}) \mu(x) dx.
\end{align*}
Taking the limit of both sides and using \cref{log_ints,lem:comparison_limit},
we obtain
\begin{equation*}
\sum_{i' \neq j'}  \frac{h_{i'j'}(0)}{\|y_{i'} - y_{j'}\|}
\geq
\limsup_{\eta \to \infty} \log(1+e^{d_\eta(y_j) - d_\eta(y_i)})  \frac{h_{ij}(0)}{\|y_i - y_j\|} \,,
\end{equation*}
showing that $d_\eta(y_j) - d_\eta(y_i)$ is bounded above for all $(i, j)$ for which $h_{ij}(0) > 0$.
By \cref{lem:connected}, the graph on $[n]$ with edge set $\{(i, j): h_{ij}(0) > 0\}$ is connected, so for any $(i, j) \in [n]^2$ we may find a path $(k_l)_{l=1}^L$ such that $k_1 = i$ and $k_L = j$, and $d_\eta(y_{k_{l + 1}}) - d_\eta(y_{k_l})$ is bounded above for all $l =1,..., L - 1$; as a result, we conclude that in fact $d_\eta(y_j) - d_\eta(y_i)$ is bounded for all $(i, j) \in [n]^2$.
Finally, since $\E_\nu d_\eta = 0$, we conclude that $d_\eta$ is bounded.
\end{proof}
All that remains is to prove the Lemma.
\begin{proof}[Proof of \cref{lem:log_subadditive}]
Fix $b \in [0,1]$.
Then~\eqref{eq:log_subadditive} holds for $a = 0$, and the derivative of the left side in $a$ is
\begin{equation*}
\frac{b}{1 + ab} \geq \frac{b}{1+a}\,,
\end{equation*}
whereas the derivative of the right side in $a$ is
\begin{equation*}
\frac{\log(1+b)}{1+a} \leq \frac{b}{1+a}\,.
\end{equation*}
We obtain that~\eqref{eq:log_subadditive} therefore holds for all $a \geq 0$.
\end{proof}

We now turn to the proof of the theorem.
The boundedness of $d_\eta$ allows us to extract a convergent subsequence, and by passing to the limit we obtain strong convexity of~\eqref{d_eq} in the limit around $0$.
\begin{proof}[Proof of \cref{dual_convergence}]
As above, we may assume $n \geq 2$.
We will show that for any sequence $(\eta_s)_{s \geq 1}$, there exists a subsequence along which $d_\eta \to 0$.
Let us fix such a sequence.

Since $d_\eta$ is bounded, by passing to a subsequence---which we again denote by $\eta_s$---we may assume that $d_\eta$ tends to a limit $d_\infty$.

Now, fix an $\epsilon>0$.
Recall from \cref{sec:prelim} that $S_{ij}(\eps)$ is the subset of $S_i$ on which $\Delta_{ik} \geq \eps$ for all $k \neq i, j$.
By definition, then, the sets $S_{ij}(\eps) \cap \{x \in S_i: \Delta_{ij} < \eps\}$ for $j \neq i$ are disjoint subsets of $S_i$.
We can therefore decompose the integral over $S_i$ into these sets to obtain
\begin{align*}
\Phi(\eta) & = \sum_{i=1}^n \int_{S_i} \log(1 + \sum_{k \neq i} e^{d_\eta(y_k) - d_\eta(y_i) - \eta \Delta_{ik}(x)}) \mu(x) dx \\
& \geq \sum_{i = 1}^n \sum_{j \neq i} \int_{S_{ij}(\eps) \cap \{x \in S_i: \Delta_{ij} < \eps\}} \log(1 + \sum_{k \neq i} e^{d_\eta(y_k) - d_\eta(y_i) - \eta \Delta_{ik}(x)}) \mu(x) dx \\
& \geq \sum_{i=1}^n \sum_{j \neq i} \int_{S_{ij}(\eps) \cap \{x \in S_i: \Delta_{ij} < \eps\}} \log(1 + e^{d_\eta(y_j) - d_\eta(y_i) - \eta \Delta_{ij}(x)}) \mu(x) dx\,.
\end{align*}
Multiplying by $\eta$ and taking the limit using~\cref{log_ints_si} yields for $\eps$ sufficiently small
\begin{equation*}
\liminf_{s \to \infty} \eta_s \Phi(\eta_s) \geq 
\sum_{i \neq j} - \Li_2(-e^{d_\infty(y_j) - d_\infty(y_i)}) \frac{h_{ij}(0; \eps)}{2\|y_i - y_j\|}\,.
\end{equation*}
Letting $\epsilon \to 0$ and applying \cref{continuous}, we obtain
\begin{equation*}
\liminf_{s \to \infty} \eta_s \Phi(\eta_s) \geq \sum_{i\neq j} - \Li_2(-e^{d_\infty(y_j) - d_\infty(y_i)})\frac{h_{ij}(0)}{2\|y_i - y_j\|}\,.
\end{equation*}
Since $h_{ij}(0) = h_{ji}(0)$ by \cref{lem:connected}, we may symmetrize this sum to obtain
\begin{equation*}
\liminf_{s \to \infty} \eta_s \Phi(\eta_s) \geq \sum_{i \neq j} \frac 12 \left[ - \Li_2(-e^{d_\infty(y_j) - d_\infty(y_i)}) - \Li_2(-e^{d_\infty(y_i) - d_\infty(y_j)})\right] \frac{h_{ij}(0)}{2\|y_i - y_j\|}\,.
\end{equation*}
By the inversion formula for the dilogarithm function~\cite[A.2.1(5)]{lewin_1981},
\begin{equation*}
\frac 12 \left[  - \Li_2(-e^{d_\infty(y_j) - d_\infty(y_i)}) - \Li_2(-e^{d_\infty(y_i) - d_\infty(y_j)})\right] =  \frac{\zeta(2)}{2} + \frac 14 (d_\infty(y_j) - d_\infty(y_i))^2\,.
\end{equation*}

Combined with \cref{lem:comparison_limit}, we conclude
\begin{align*}
\frac{\zeta(2)}{4} \sum_{i \neq j} \frac{h_{ij}(0)}{\|y_i - y_j\|} & \geq \limsup_{s \to \infty} \eta_s \Phi(\eta_s) \\
& \geq \liminf_{s \to \infty} \eta_s \Phi(\eta_s) \\
& \geq \frac{\zeta(2)}{4} \sum_{i \neq j} \frac{h_{ij}(0)}{\|y_i - y_j\|} + \frac 18 \sum_{i \neq j} (d_\infty(y_j) - d_\infty(y_i))^2 \frac{h_{ij}(0)}{\|y_i - y_j\|}\,,
\end{align*}
implying that $d_\infty(y_j) = d_\infty(y_i)$ if $h_{ij}(0) \neq 0$, and that
\begin{equation}\label{phi_limit}
\lim_{\eta \to \infty} \eta \Phi(\eta) = \frac{\zeta(2)}{4} \sum_{i \neq j} \frac{h_{ij}(0)}{\|y_i - y_j\|} = \frac{\zeta(2)}{2} \sum_{i < j} \frac{h_{ij}(0)}{\|y_i - y_j\|}\,.
\end{equation}
We conclude as in the proof of \cref{dual_bounded}.
\end{proof}

\Cref{cost_cor} is immediate in light of \eqref{phi_limit}, \eqref{eq:phi_value}, and \eqref{ent_shift}.

\section{Convergence of the suboptimality}\label{sec:quadratic_sub_opt}

In this section we prove our main result, from which Theorem~\ref{thm:subopt_limit} follows.
\begin{theorem}\label{thm:main}
Under Assumptions~\ref{assume} and~\ref{continuous},
\begin{equation*}
\lim_{\eta \to \infty} \eta^2(\E_{\pi_\eta}[\|x - y\|^2] - \E_{\pi^*}[\|x - y\|^2]) = \frac{\zeta(2)}{2}\sum_{i < j} \frac{h_{ij}(0)}{\|y_i - y_j\|}\,.
\end{equation*}
\end{theorem}
\par The proof uses two lemmas. The first lemma decomposes the suboptimality of an arbitrary coupling $\pi \in \Pi(\mu,\nu)$ into a sum of nonnegative terms involving the slack operators $\Delta_{ij}$.

\begin{lemma}[Suboptimality decomposition]\label{lem:nonneg_decomp} For any $\pi \in \Pi(\mu, \nu)$,
	\begin{align}
	\label{eq:lem-nonneg_decomp}
	\E_{\pi}[\|x - y\|^2] - \E_{\pi^*}[\|x - y\|^2]  = \sum_{i \neq j}
	\int_{S_i} \Delta_{ij}(x) d\pi(x, y_j)\,.
	\end{align}
\end{lemma}

The second lemma explicitly computes the integrals that result from using this decomposition on the coupling $\pi_{\eta}$.
We recall the notation $d_\eta = \eta(g_\eta - g^*)$ from \cref{sec:dual}.

\begin{lemma}[Sigmoid slack integrals]\label{lem:subopt_int2}
	Under Assumptions~\ref{assume} and~\ref{continuous}, for any $i \neq j$,
	\begin{align}
	\lim_{\eta \to \infty} \eta^2 \int_{S_i} \frac{\Delta_{ij}(x)  e^{d_\eta(y_j) -\eta \Delta_{ij}(x) }}{\sum_{k} e^{d_\eta(y_k) -\eta \Delta_{ik}(x)}} \mu(x) dx
	=
	\frac{\zeta(2) h_{ij}(0)}{4 \|y_i - y_j\|}.
	\label{eq:lemsubopt_int2}
	\end{align}
\end{lemma}
\begin{proof}
First,
\begin{equation*}
\lim_{\eta \to \infty} \eta^2 \int_{S_i} \frac{\Delta_{ij}(x)  e^{d_\eta(y_j) -\eta \Delta_{ij}(x) }}{\sum_{k} e^{d_\eta(y_k) -\eta \Delta_{ik}(x)}} \mu(x) dx \leq \lim_{\eta \to \infty} \eta^2 \int_{S_i} \frac{\Delta_{ij}(x)  e^{d_\eta(y_j) - d_\eta(y_i) -\eta \Delta_{ij}(x) }}{1 +  e^{d_\eta(y_j) - d_\eta(y_i) -\eta \Delta_{ij}(x) }} \mu(x) dx\,,
\end{equation*}
and since $d_\eta \to 0$ by \cref{dual_convergence}, we can apply~\cref{log_int_subopt} to conclude that the limit is bounded above by
\begin{equation*}
-\Li_2(-1) \frac{h_{ij}(0)}{2 \|y_i - y_j\|} = \frac{\zeta(2) h_{ij}(0)}{4 \|y_i - y_j\|}\,.
\end{equation*}

On the other hand, for any $\eps > 0$ and $c > 1$, we have
\begin{multline*}
\lim_{\eta \to \infty} \eta^2 \int_{S_i} \frac{\Delta_{ij}(x)  e^{d_\eta(y_j)-\eta \Delta_{ij}(x) }}{\sum_{k} e^{d_\eta(y_k) -\eta \Delta_{ik}(x)}} \mu(x) dx  \geq \lim_{\eta \to \infty} \eta^2 \int_{S_{ij}(\eps)} \frac{\Delta_{ij}(x)  e^{d_\eta(y_j) -\eta \Delta_{ij}(x) }}{\sum_{k} e^{d_\eta(y_k) -\eta \Delta_{ik}(x)}} \mu(x) dx \\
 \geq \lim_{\eta \to \infty} \eta^2 \int_{S_{ij}(\eps)} \frac{\Delta_{ij}(x)  e^{d_\eta(y_j) -\eta \Delta_{ij}(x) }}{e^{d_\eta(y_i)} + (n-2) e^{2 \|d_\eta\|_\infty - \eta \eps} + e^{d_\eta(y_j)  -\eta \Delta_{ij}(x) }} \mu(x) dx \\
 \geq \lim_{\eta \to \infty} \eta^2 \int_{S_{ij}(\eps)} \frac{\Delta_{ij}(x)  e^{d_\eta(y_j) -\eta \Delta_{ij}(x) }}{c + e^{d_\eta(y_j)  -\eta \Delta_{ij}(x) }} \mu(x) dx\,,
\end{multline*}
where we have used the fact that $d_\eta \to 0$, so that $e^{d_\eta(y_i)} + (n-2) e^{2 \|d_\eta\|_\infty - \eta \eps} < c$ for all $\eta$ sufficiently large.
By \cref{log_int_subopt}, for $\eps$ sufficiently small, this limit is
\begin{equation*}
-\Li_2(-1/c) \frac{h_{ij}(0; \eps)}{2 \|y_i - y_j\|}\,,
\end{equation*}
and taking $c \to 1$ and $\eps \to 0$ and applying \cref{continuous}, we obtain that the limit is also bounded below by
\begin{equation*}
\frac{\zeta(2) h_{ij}(0)}{4 \|y_i - y_j\|}\,,
\end{equation*}
completing the proof.
\end{proof}

With these two lemmas in hand, the proof of Theorem~\ref{thm:subopt_limit} follows readily. 

\begin{proof}[Proof of Theorem~\ref{thm:subopt_limit}] 
	By Lemma~\ref{lem:nonneg_decomp} and~\eqref{coupling_with_d}.
	\begin{align*}
		\lim_{\eta \to \infty} \eta^2(\E_{\pi_\eta}[\|x - y\|^2] - \E_{\pi^*}[\|x - y\|^2])
		=  \lim_{\eta \to \infty}  \sum_{i \neq j} \eta^2 \int_{S_i} \Delta_{ij}(x) \frac{e^{d_{\eta}(y_j) - \eta \Delta_{ij}(x)}}{\sum_{k} e^{d_{\eta}(y_k) - \eta \Delta_{ik}(x)}} \mu(x) dx\,.
	\end{align*}
	By \cref{lem:subopt_int2}, this is equal to
	\begin{align*}
		\frac{\zeta(2)}{4} \sum_{i \neq j} \frac{h_{ij}(0)}{\|y_i - y_j\|}.
	\end{align*}
	Summing over $i \neq j$ and using the symmmetry
	\begin{equation*}
		\frac{h_{ij}(0)}{\|y_i - y_j\|} = \frac{h_{ji}(0)}{\|y_j - y_i\|}
	\end{equation*}
	finishes the proof.
\end{proof}

It now suffices to prove Lemma~\ref{lem:nonneg_decomp}.

\begin{proof}[Proof of \cref{lem:nonneg_decomp}]
	By strong duality and the fact that $\pi \in \Pi(\mu, \nu)$,
	\begin{equation*}
		\E_{\pi^*}[\|x - y\|^2] = \E_{\mu} f^* + \E_\nu g^* = \E_\pi[f^*(x) + g^*(y)]\,.
	\end{equation*}
	Therefore
	\begin{align*}
		\E_{\pi}[\|x - y\|^2] - \E_{\pi^*}[\|x - y\|^2] &=
		\E_{\pi}[\|x - y\|^2 - f^*(x) - g^*(y)] \\
		&= \sum_{i, j} \int_{S_i} [\|x - y_j\|^2 - f^*(x) - g^*(y_j)] d \pi(x, y_j) \\
		& = \sum_{i, j} \int_{S_i} \Delta_{ij}(x) d \pi(x, y_j)\,,
	\end{align*}
	where the last step uses the definition of $\Delta_{ij}$~\eqref{defn:slack}.
	Since $\Delta_{ij}(x) = 0$ if $i = j$, the diagonal terms vanish, proving the claim.
\end{proof}

\section{Supplementary results}\label{sec:appendix}
This section collects several supplementary lemmas relating to the integration of relevant quantities depending on the slacks in the cell $S_i$, as well as the proofs of two technical claims from \cref{sec:prelim}.

\subsection{The dilogarithm function}
The properties of our asymptotic expansion---including the presence of the constant $\zeta(2)/2$---rely on several classical properties of the dilogarithm function.
The claims below appear in \cite{lewin_1981}.

\begin{defn}
The \emph{dilogarithm} function is given by
\begin{equation*}
\Li_2(z) = \sum_{s = 1}^\infty \frac{z^s}{s^2} \quad |z| \leq 1
\end{equation*}
and extended to $\mathbb C \setminus (1, \infty)$ by analytic continuation.
\end{defn}

An immediate consequence of this definition is the special value
\begin{equation}\label{li1}
\Li_2(-1) = \sum_{s=1}^\infty \frac{(-1)^s}{s^2} = - \frac{\zeta(2)}{2} = - \frac{\pi^2}{12}.
\end{equation}
Moreover, the analyticity of $\Li_2$ away from the branch cut implies in particular that it is continuous on the negative reals.

The appearance of the dilogarithm in our proofs follows directly from two of its integral representations, which arise naturally from the solutions of the entropic optimal transport problem in the semi-discrete setting studied in this paper.

\begin{lemma}[\cite{lewin_1981}]\label{log_ints}
The dilogarithm satisfies
\begin{equation*}
-\Li_2(-1/c) = \int_0^\infty \frac{t e^{-t}}{c + e^{-t}}\, dt = \int_0^\infty \log(1 + e^{-t}/c)\, dt
\end{equation*}
for all $c > 0$.
In particular,
\begin{equation*}
-\Li_2(-1)= \int_0^\infty \frac{t e^{-t}}{1 + e^{-t}}\, dt = \int_0^\infty \log(1 + e^{-t})\, dt = \frac{\zeta(2)}{2}\,.
\end{equation*}
\end{lemma}

Rather than using \cref{log_ints} directly, we will typically be integrating with respect to the measure $\mu$ over a power cell.
However, as the following lemmas show, in the large-$\eta$ limit we can still employ the integral identities of \cref{log_ints} to obtain explicit expressions in terms of the dilogarithm.

\begin{lemma}\label{log_ints_si}
Let $M_\eta$ be such that $\lim_{\eta \to \infty} M_\eta = M > 0$, and let $a > 0$ be small enough that \cref{continuous} holds.
Then
\begin{equation*}
\lim_{\eta \to \infty} \eta \int_{S_{ij}(a)} \log(1 + M_\eta e^{-\eta \Delta_{ij}(x)}) \mu(x) dx = -\Li_2(-M) \frac{h_{ij}(0; a)}{2 \|y_i - y_j\|}\,.
\end{equation*}
The same claim holds if $S_{ij}(a)$ is replaced by $S_{ij}(a) \cap \{x \in S_i: \Delta_{ij}(x) < a\}$.
\end{lemma}
\begin{proof}
By a change of variables, we can write
\begin{equation*}
\eta \int_{S_{ij}(a)} \log(1 + M_\eta e^{-\eta \Delta_{ij}(x)}) \mu(x) dx = \frac{\eta }{2 \|y_i - y_j\|} \int_0^\infty \log(1 + M_\eta e^{-\eta t}) h_{ij}(t; a) dt\,.
\end{equation*}
Since $M_\eta$ tends to a limit, it is bounded, and so for any $\eps > 0$ the function $\eta \log(1 + M_\eta e^{-\eta t})$ tends uniformly to $0$ on $[\eps, \infty)$.
Since $h_{ij}(t; a) \in L_1$, this implies that
\begin{equation*}
\lim_{\eta \to \infty} \eta \int_\eps^\infty \log(1 + M_\eta e^{-\eta t}) h_{ij}(t; a) dt = 0\,.
\end{equation*}
The integral therefore only depends on an interval near zero; in particular, replacing the set $S_{ij}(a)$ by $S_{ij}(a) \cap \{x \in S_i: \Delta_{ij}(x) < a\}$, which has the effect of integrating from $0$ to $a$ instead of $0$ to $\infty$, does not affect the value of the limit.

A second change of variables gives
\begin{equation*}
\lim_{\eta \to \infty} \frac{\eta }{2 \|y_i - y_j\|} \int_0^\eps \log(1 + M_\eta e^{-\eta t}) h_{ij}(t; a) dt = \lim_{\eta \to \infty} \frac{1 }{2 \|y_i - y_j\|}\int_0^{\eta \eps} \log(1 + M_\eta e^{-t}) h_{ij}(\eta^{-1} t; a)dt\,.
\end{equation*}
Let us first consider replacing $h_{ij}(\eta^{-1} t; a)$ by $h_{ij}(0, a)$.
Dominated convergence and~\cref{log_ints} then imply
\begin{equation*}
\lim_{\eta \to \infty} \frac{h_{ij}(0; a)}{2 \|y_i - y_j\|}\int_0^{\eta \eps} \log(1 + M_\eta e^{-t}) dt = \frac{h_{ij}(0; a)}{2 \|y_i - y_j\|} \int_0^\infty \log(1 + M e^{-t}) dt = - \Li_2(-M) \frac{h_{ij}(0; a)}{2 \|y_i - y_j\|}\,,
\end{equation*}
which is the desired limit.

It therefore suffices to show that replacing $h_{ij}(\eta^{-1} t; a)$ by $h_{ij}(0, a)$ is justified.
If we make this replacement, we incur an error of size at most
\begin{equation*}
\sup_{\delta \leq \eps} |h_{ij}(\delta; a) - h_{ij}(0, a)| \frac{1}{2 \|y_i - y_j\|}\int_0^{\eta \eps} \log(1 + M_\eta e^{-t}) dt\,.
\end{equation*}
Since the integral is bounded and $h_{ij}(t; a)$ is continuous at $t = 0$ (Assumption~\ref{continuous}), this error vanishes as $\eps \to 0$, completing the proof.
\end{proof}

\begin{lemma}\label{log_int_subopt}
Let $M_\eta$ be such that $\lim_{\eta \to \infty} M_\eta = M > 0$, let $a \geq 0$ be small enough that \cref{continuous} holds, and let and $c > 0$ be arbitrary.
Then
\begin{equation*}
\lim_{\eta \to \infty} \eta^2 \int_{S_{ij}(a)} \frac{\Delta_{ij}(x)  M_\eta e^{-\eta \Delta_{ij}(x) }}{c +  M_\eta e^{-\eta \Delta_{ij}(x) }} \mu(x) dx = -\Li_2(-M/c) \frac{h_{ij}(0; a)}{2 \|y_i - y_j\|}\,.
\end{equation*}
\end{lemma}
\begin{proof}
The proof is exactly analogous to that of \cref{log_ints_si}.
Fix $\eps > 0$.
First, by change of variables and the uniform convergence of $\frac{\eta^2 t M_\eta e^{-\eta t}}{c + M_\eta e^{-\eta t}}$ to $0$ on $[\eps, \infty)$, it suffices to evaluate
\begin{equation*}
\lim_{\eta \to \infty} \frac{1}{2 \|y_i - y_j\|} \eta^2 \int_0^\eps \frac{t M_\eta e^{-\eta t}}{c + M_\eta e^{-\eta t}} h_{ij}(t; a) dt = \lim_{\eta \to \infty} \frac{1}{2 \|y_i - y_j\|} \int_0^{\eps \eta} \frac{t M_\eta e^{- t}}{c + M_\eta e^{- t}} h_{ij}(\eta^{-1} t; a)  dt\,.
\end{equation*}
As above, replacing $h_{ij}(\eta^{-1} t; a)$ by $h_{ij}(0; a)$ incurs error that vanishes as $\eps \to 0$.
We obtain that the desired limit is
\begin{equation*}
\lim_{\eta \to \infty} \frac{h_{ij}(0; a)}{2 \|y_i - y_j\|} \int_0^{\eps \eta} \frac{t M_\eta e^{- t}}{c + M_\eta e^{- t}} dt\,.
\end{equation*}
By dominated convergence and \cref{log_ints}, this is
\begin{equation*}
- \Li_2(-M/c)\frac{h_{ij}(0; a)}{2 \|y_i - y_j\|}\,,
\end{equation*}
as desired.
\end{proof}

\subsection{Proof of \cref{compact}}
The proof is inspired by~\cite[Lemma 46]{merigot2020optimal}.
For any $i \neq j$, define the hyperplane
\begin{equation*}
H_{ij} = \{x \in \R^d: 2 \langle x, y_i - y_j\rangle - \|y_i\|^2 + \|y_j\|^2 - g_j^* + g_i^* = 0\}\,.
\end{equation*}
We require the following lemma.
\begin{lemma}\label{hyperplanes}
If $g^*$ is optimal, then $H_{ij} \neq  H_{ik}$ for all $j \neq k$.
\end{lemma}
\begin{proof}
Suppose that $H_{ik}$ and $H_{ij}$ coincide for some $j \neq k$.
Then the definition of $H_{jk}$ implies that it coincides with $H_{ik}$ and $H_{ij}$ as well.
The cells $S_i$, $S_j$, and $S_k$ are convex sets with positive $\mu$ (and hence positive Lebesgue) measure; therefore, they have non-empty interiors.
If we consider the two open halfspaces defined by the hyperplane $H_{ij} = H_{ik} = H_{jk}$, then there exist two of the cells---say, $S_i$ and $S_j$---whose interiors lie in the same open halfspace.
But this contradicts the fact that $2 \langle x, y_i - y_j\rangle - \|y_i\|^2 + \|y_j\|^2 - g_j^* + g_i^* > 0$ for all $x \in \mathrm{int}(S_i)$, and $2 \langle x, y_i - y_j\rangle - \|y_i\|^2 + \|y_j\|^2 - g_j^* + g_i^* < 0$ for all $x \in \mathrm{int}(S_j)$.
So $H_{ik}$ and $H_{ij}$ cannot coincide, as claimed.
\end{proof}

Let us fix an $a \geq 0$ sufficiently small and prove the continuity of $t \mapsto h_{ij}(t; a)$.
Given a nonnegative sequence $t_n \to 0$, consider
\begin{align*}
h_{ij}(t_n; a) - h_{ij}(0; a) & = \int_{H_{ij}(t_n; a)} \mu(x) d\mathcal H_{d-1}(x) - \int_{H_{ij}(0; a)} \mu(x) d\mathcal H_{d-1}(x) \\
& = \int_{ H_{ij}} (\mathbbold{1}[x + t_n v \in S_{ij}(a)] \mu(x + t_n v) - \mathbbold{1}[x \in S_{ij}(a)] \mu(x))d\mathcal H_{d-1}(x).
\end{align*}
Continuity of $\mu$ implies that $\mu(x + t_n) \to \mu(x)$ pointwise.
We will now show that $\mathbbold{1}[x + t_n v \in S_{ij}(a)] \to \mathbbold{1}[x \in S_{ij}(a)]$ for $\mathcal H_{d-1}$-almost every $x$.
First, since $S_{ij}(a)$ is closed, if $x \notin S_{ij}(a)$ then $x \notin S_{ij}(a) - t_n v$ for all $t_n$ sufficiently close to $0$.
Thus, $\limsup_{n \to \infty} \mathbbold{1}[x + t_n v \in S_{ij}(a)] \leq \mathbbold{1}[x \in S_{ij}(a)]$.

On the other hand, the set $S_{ij}(a)$ is a convex set defined by the constraints
\begin{align*}
2 \langle x, y_i - y_j\rangle - \|y_i\|^2 + \|y_j\|^2 - g_j^* + g_i^* & \geq 0\\
2 \langle x, y_i - y_k\rangle - \|y_i\|^2 + \|y_k\|^2 - g_k^* + g_i^* & \geq a \quad \quad \forall k \neq i ,j\,.
\end{align*}
By \cref{hyperplanes}, for all $a$ sufficiently small and all $k \neq i, j$, the intersection of $H_{ij}$ and $\{x \in \R^d: 2 \langle x, y_i - y_k\rangle - \|y_i\|^2 + \|y_k\|^2 - g_k^* + g_i^* = a\}$ has codimension at least $2$.
Therefore, for $\mathcal H_{d-1}$-almost every $x \in S_{ij}(a) \cap H_{ij}$,
\begin{align*}
2 \langle x, y_i - y_j\rangle - \|y_i\|^2 + \|y_j\|^2 - g_j^* + g_i^* & = 0 \\
2 \langle x, y_i - y_k\rangle - \|y_i\|^2 + \|y_k\|^2 - g_k^* + g_i^* & > a \quad \quad \forall k \neq i ,j\,.
\end{align*}
For such $x$, we therefore have that $x + t_n v \in S_{ij}(a)$ for $t_n$ sufficiently close to $0$, and $\liminf_{n \to \infty} \mathbbold{1}[x + t_n v \in S_{ij}(a)] \geq \mathbbold{1}[x \in S_{ij}(a)]$.
Therefore, $\mathbbold{1}[x + t_n v \in S_{ij}(a)] \to \mathbbold{1}[x \in S_{ij}(a)]$ for $\mathcal H_{d-1}$-almost every $x$.

Since $\mu$ is dominated along hyperplanes, $\mathbbold{1}[x + t_n v \in S_{ij}(a)] \mu(x + t_n v) - \mathbbold{1}[x \in S_{ij}(a)] \mu(x)$ is dominated by an integrable function on $H_{ij}$, and the claim follows.

The second argument is simpler: given a sequence $a_n \to 0$, we have
\begin{equation*}
h_{ij}(0; a_n) - h_{ij}(0; 0) = \int_{H_{ij}} (\mathbbold{1}[x \in S_{ij}(a_n)] - \mathbbold{1}[x \in S_{i}]) \mu(x) d \mathcal H_{d-1}(x)\,.
\end{equation*}
Since $S_{ij}(a_n) \subseteq S_i$, it is clear that $\limsup_{n \to \infty}\mathbbold{1}[x \in S_{ij}(a_n)] \leq \mathbbold{1}[x \in S_{i}]$.
And as above, $\mathcal H_{d-1}$-almost every $x \in S_{ij} \cap H_{ij}$ satisfies
\begin{align*}
2 \langle x, y_i - y_j\rangle - \|y_i\|^2 + \|y_j\|^2 - g_j^* + g_i^* & = 0 \\
2 \langle x, y_i - y_k\rangle - \|y_i\|^2 + \|y_k\|^2 - g_k^* + g_i^* & > 0 \quad \quad \forall k \neq i ,j\,.
\end{align*}
and for these $x$, $\liminf_{n \to \infty}\mathbbold{1}[x \in S_{ij}(a_n)] \geq \mathbbold{1}[x \in S_{i}]$.
This proves the claim.
\qed

\subsection{Proof of \cref{lem:connected}}
That $h_{ij}(0) = h_{ji}(0)$ follows from the fact that $H_{ij}(0) = H_{ji}(0) = S_i \cap S_j$.

Now, we show that the graph with edge set $\{(i, j): h_{ij}(0) > 0\}$ is connected.
Since $\mu$ is positive on the interior of its support, if $h_{ij}(0) = 0$, then $\mathrm{int}(\mathrm{supp}(\mu)) \cap (S_i \cap S_j)$ has zero $\mathcal H_{d-1}$ measure.
By~\cite[Lemma 49]{merigot2020optimal}, this implies that the set
\begin{equation*}
Z := \mathrm{int}(\mathrm{supp}(\mu)) \setminus \left( \bigcup_{ij \colon h_{ij}(0) = 0} S_i \cap S_j\right)
\end{equation*}
is path connected.

Now, suppose that the graph has $K$ connected components.
For each component $C_k \subseteq [n]$, let
\begin{equation*}
Z_k = \bigcup_{i \in C_k} (Z \cap S_i)\,.
\end{equation*}
Since each cell $S_i$ is closed and has positive $\mu$ mass, each $Z_k$ is nonempty and closed in the subspace topology on $Z$.
Moreover, they are disjoint by the definition of $Z$.
Therefore the $Z_k$ form a non-empty, closed partition of the connected set $Z$, so $K = 1$.
\qed

\footnotesize
\bibliographystyle{bib_stuff/IEEEannot}
\bibliography{bib_stuff/annot}

\normalsize

\end{document}